\documentclass{article}
\pdfoutput=1
\usepackage[T1]{fontenc}
\usepackage{amsmath}
\usepackage{amssymb}
\usepackage{graphicx}
\usepackage{color}
\usepackage{bbold}
\usepackage{dsfont}

\renewcommand{\P}{{\rm P}}
\newcommand{\E}{\mathbb{E}}

\newcommand{\NN}{\mathbb{N}}
\renewcommand{\P}{\mathbb{P}}

\newcommand{\tr}{^{\textnormal {\tiny  T}}}

\newcommand{\diag}{\mathrm{diag}}

\newcommand{\ic}{\mathrm{\bf{i}}}

\newcommand{\ca}{\mathcal{A}}
\newcommand{\cb}{\mathcal{B}}
\newcommand{\EE}{\mathcal{E}}

\newcommand{\PP}{\mathcal{P}}

\newcommand{\vect}[1]{\boldsymbol #1}

\newcommand{\vpi}{\vect \pi}
\newcommand{\valpha}{\vect \alpha}

\newcommand{\vgamma}{\vect \gamma}

\newcommand{\vlambda}{\vect \lambda}

\newcommand{\vone}{\vect 1}
\newcommand{\vzero}{\vect 0}

\newcommand{\ve}{\vect e}

\newcommand{\vg}{\vect g}

\newcommand{\vh}{\vect h}

\newcommand{\vr}{\vect r}
\newcommand{\vu}{\vect u}

\newcommand{\vv}{\vect v}

\newcommand{\vx}{\vect x}

\newcommand{\vy}{\vect y}

\newcommand{\whE}{\widehat{E}}
\newcommand{\whG}{\widehat{G}}
\newcommand{\whH}{\widehat{H}}
\newcommand{\whM}{\widehat{M}}
\newcommand{\whR}{\widehat{R}}

\newcommand{\wheta}{\widehat{\eta}}

\newcommand{\whchi}{\widehat{\chi}}
\newcommand{\whc}{\widehat{c}}

\newcommand{\whu}{\widehat{\vu}}
\newcommand{\whv}{\widehat{\vv}}
\newcommand{\whpi}{\widehat{\pi}}

\newcommand{\vligne}[1]{\begin{bmatrix} #1 \end{bmatrix}}

\newtheorem{defn}{Definition}[section]
\newtheorem{lem}[defn]{Lemma}

\newtheorem{thm}[defn]{Theorem}
\newtheorem{cor}[defn]{Corollary}
\newtheorem{rem}[defn]{Remark}

\newcommand{\qed}{\hfill $\square$}
\newenvironment{proof}{
      \noindent {\bf Proof }}{\qed
      \vspace{0.25\baselineskip}
}
\newcommand{\debproof}{\begin{proof}}
\newcommand{\finproof}{\end{proof}}

\definecolor{darkmagenta}{rgb}{0.5,0,0.5}
\definecolor{darkgreen}{rgb}{0,0.6,0}
\definecolor{darkblue}{rgb}{0,0,0.6}
\definecolor{darkred}{rgb}{0.8,0,0}
\definecolor{mellow}{rgb}{.847, 0.72, 0.525}

\usepackage{pifont}

\usepackage[normalem]{ulem}

\begin{document}

\title{
Traffic lights, clumping and QBDs}
\author{Steven Finch\thanks{MIT Sloan School of Management, Cambridge,
    MA, USA, \texttt{steven\_finch@harvard.edu} 
}
\and Guy Latouche\thanks{Universit\'e libre de Bruxelles,  Facult\'e des sciences, CP212, Boulevard du Triomphe 2, 1050 Bruxelles,
  Belgium,        \texttt{latouche@ulb.ac.be}} 
\and Guy Louchard\thanks{Universit\'e libre de Bruxelles,  Facult\'e des sciences, CP212, Boulevard du Triomphe 2, 1050 Bruxelles,
  Belgium,        \texttt{louchard@ulb.ac.be}} 
\and Beatrice Meini\thanks{Universita di Pisa, Dipartimento di
  Matematica, 56127 Pisa, Italy,  \texttt{meini@dm.unipi.it}}
}
\date{July 9, 2019}

\maketitle

\begin{abstract}
In discrete time, $\ell$-blocks of red lights are separated by $\ell$-blocks
of green lights. Cars arrive at random. \ We seek the distribution of maximum
line length of idle cars, and justify conjectured probabilistic asymptotics
algebraically for $2\leq\ell\leq3$ and numerically for $\ell\geq4$.
\end{abstract}

\section{Introduction}
\label{s:introduction}

Cars arrive at a traffic light according to a Bernoulli process:
during each unit of time, one car arrives or no cars arrive,
respectively with probability $p$ and $q=1-p$; we assume throughout
that $p<q$.  The traffic light alternates
between being red and green during intervals of time of length $\ell$.
When the traffic light is red, arriving cars wait and form a line,;
when the traffic light is green, one car, if any are present, may pass
through during each unit of time.  If the line is empty at some time
$t$ while the light is green, and if a new car arrives, then the
arriving car passes through immediately.

Assorted expressions emerge for this problem (Finch and
Louchard~\cite{FL1-heu,
fl18}).  Let $X_{1}$,
$X_{2}$, \ldots be a sequence of independent random variables
satisfying%
\begin{align*}
& \mathbb{P}\left\{  X_{i}=1\right\}  =p,
& & \mathbb{P}\left\{  X_{i} =0\right\}  =q
& & \text{if }i\equiv1,2,\ldots,\ell\,\operatorname{mod}%
\,2\ell;
\\
& \mathbb{P}\left\{  X_{i}=0\right\}  =p, 
& & \mathbb{P}\left\{  X_{i}%
=-1\right\}  =q 
& & \text{if }i\equiv\ell+1,\ell+2,\ldots,2\ell
\,\operatorname{mod}\,2\ell.
\end{align*}
Define $S_{j}=\max\left\{
S_{j-1}+X_{j},0\right\}  $ for all $j \geq 1$, with $S_0$ a
non-negative integer.  The
quantity 
\[
M_{T}=\max\nolimits_{0\leq j\leq T}S_{j}
\] 
is the worst-case traffic
congestion, as opposed to the average-case often cited. 

Only the
circumstance when $\ell=1$ is amenable to rigorous treatment, as far as is
known. 
The Poisson clumping heuristic (Aldous~\cite{Ald-heu}), while not a theorem, gives
results identical to exact asymptotic expressions when such exist, and
evidently provides excellent predictions otherwise. Consider an irreducible
positive recurrent Markov chain with stationary distribution $\pi$.
{
Let $T_k$ be the time necessary to reach level $k$; 
the clumping heuristic states that 
\begin{equation}
   \label{e:MTa}
\P\left[ T_k > T \right] = 
\mathbb{P}\left[  M_{T}<k\right]  \sim \exp\left(  -\frac{\pi_{k}}%
{\mathbb{E}(C_k)} T \right)
\end{equation}
asymptotically as $k \rightarrow \infty$ and $T =\Theta(\E(C_k) /\pi_k)$,
where $C_k$ is the sojourn time in $k$ during a so-called clump of
nearby visits to~$k$
 --- the argument behind the heuristic is
briefly described in Section~\ref{s:sojourn}.
}

The traffic light process is a two-dimensional discrete-time Markov
chain $\{X(t)\} = \{S_t, \varphi_t\}_{t = 0, 1, 2, \ldots}$ where $S_t \in \NN$
is the length of the line at time $t$ and
$\varphi_t \in \{1, \ldots, 2 \ell\}$ is an indicator of the state of
the traffic light: it is red for $1 \leq \varphi_t \leq \ell$ and
green for $\ell + 1 \leq \varphi_t \leq 2\ell$.  We order the states
in lexicographic order {and decompose the state space $\EE = \{(n,i): n
\geq 0, 1 \leq i \leq 2 \ell\}$ into subsets of constant values of
$n$: $\EE = \cup_{n \geq 0} \EE_n$, with $\EE_n = \{(n,1), \ldots (n,
2\ell)\}$.  We organise the transition matrix $\PP_\ell$ in a
conformant manner, so that it}
takes the block-structure of a Quasi-Birth-and-Death (QBD) process,
that is,
\begin{equation}
   \label{e:Porig}
\PP_\ell  = \vligne{\cb & \ca_1 \\ \ca_{-1} & \ca_0 & \ca_1 \\  & \ca_{-1} & \ca_0 &
                                                                 \ddots
  \\  & & \ddots & \ddots}
\end{equation}
{where $\cb$, $\ca_{-1}$, $\ca_0$, and $\ca_1$ collect the
  transition probabilities from $\EE_0$ to $\EE_0$, and from $\EE_n$
  to $\EE_{n-1}$, $\EE_n$ and $\EE_{n+1}$, for $n \geq 1$, respectively.
}

{In detailed notation:
\begin{align*}
(\ca_k)_{ij} & = \PP[X_{t+1} = (n+k,j) | X_t = (n,i)] \qquad \mbox{for
               $n\geq 1$,}
  \\
(\cb)_{ij} & = \PP[X_{t+1} = (0,j) | X_t = (0,i)].
\end{align*}
}
{Clearly, $\cb$, $\ca_{-1}$, $\ca_0$, and $\ca_1$ are
matrices of size $2\ell$, and it is easy to verify that
they are as follows (the unspecified entries are equal to 0)
}
\begin{align*}
\ca_{-1} & =
\left[
  \begin{array}{cccc|cccc}
0 & 0 & & & \\
& \ddots & \ddots & & & \\
& & 0 & 0 \\
& & & 0 & 0 \\
\hline
& & & & 0 & q \\
& & & & & \ddots & \ddots \\
& & & & & & 0 & q \\
q & & & & & & & 0
  \end{array}
\right],
&
\ca_{0} & =
\left[
  \begin{array}{cccc|cccc}
0 & q & & & \\
& \ddots & \ddots & & & \\
& & 0 & q \\
& & & 0 & q \\
\hline
& & & & 0 & p \\
& & & & & \ddots & \ddots \\
& & & & & & 0 & p \\
p & & & & & & & 0
  \end{array}
\right]
\end{align*}
\begin{align*}
\ca_{1} & =
\left[
  \begin{array}{cccc|cccc}
0 & p & & & \\
& \ddots & \ddots & & & \\
& & 0 & p \\
& & & 0 & p \\
\hline
& & & & 0 & 0 \\
& & & & & \ddots & \ddots \\
& & & & & & 0 & 0 \\
0 & & & & & & & 0
  \end{array}
\right],
& 
\mbox{and} \quad
\cb &
= \ca_{-1} + \ca_0.
\end{align*}

We use this representation in Section \ref{s:natural} but we start by
using an approximate representation, better suited for an exact
analysis for small values of $\ell$, as shown in Appendix~\ref{s:exact2} for $\ell =
2$ and in \cite{FL1-heu, fl18} for $\ell =  2$ and 3.  
We parse the sequence
$\{S_0, S_1, \ldots\}$ over intervals of length
$2\ell$ corresponding to  cycles of $\ell$ red and
$\ell$ green units of time and we define the Markov chain $\{Z_t\}$
with $Z_t = S_{2 \ell t}$.  For such a sub-walk, we need not keep
track of $\varphi_t$; it is enough to decide on the value of
$\varphi_0$.  The transition probabilities are then determined by the
proper product of the transition matrices of infinite size
\[
U = \vligne{
q & p & \cdot & \cdot & \cdot & \\
\cdot & q & p & \cdot & \cdot & \\
\cdot & \cdot & q & p & \cdot & \\
\cdot & \cdot & \cdot & q & p & \\
\cdot & \cdot & \cdot & \cdot & q & \ddots \\
& & & & & \ddots
}
\qquad \mbox{and} \qquad
V = \vligne{
1 & \cdot & \cdot & \cdot & \cdot & \\
q & p & \cdot & \cdot & \cdot & \\
\cdot & q & p & \cdot & \cdot & \\
\cdot & \cdot & q & p & \cdot & \\
\cdot & \cdot & \cdot & q& p & \ddots \\
& & & & & \ddots
},
\]
where $U$ is the one-step transition probability matrix when the
traffic light is red and $V$  is the matrix when the light is green.

{For the first sub-walk, we} assume that $\varphi_0 =1$, and so the transition matrix of
$\{Z_t\}$ is $P_\ell = U^\ell V^\ell$.  {We briefly consider in
  Appendix~\ref{s:exact2} the sub-walk with $\varphi_0 = \ell +1$ and transition
  matrix $V^\ell U^\ell$.}
 In detailed expression, $(P_\ell)_{n,n+j} = p_j$ for
$n+j \geq 1$ and $(P_\ell)_{n,0} = p_{-n} + \cdots +
p_{-\ell}$  for $n \leq \ell$, with 
\begin{equation}
   \label{e:pj}
p_j = {2 \ell \choose \ell+j} p^{\ell+j} q^{\ell - j}, \qquad
-\ell \leq j \leq \ell
\end{equation}
being the probability that the line in front of the street light increases (if $j >0$) or
decreases (if $j<0$) by $|j|$ cars during any cycle of length $2 \ell$,

For fixed $\ell\geq1$, we have the following conjecture
\cite{FL1-heu}:%
\begin{align}
   \label{e:MT}
& \mathbb{P}\left[  M_{T}\leq\log_{q^{2}/p^{2}}(T)+h\right]  \sim\exp\left[
-\frac{\chi_{\ell}(p)}{2\ell}\left(  \frac{q^{2}}{p^{2}}\right)  ^{-h}\right],
\end{align}
\[
 \mathbb{E}\left(  M_{T}\right)  \sim\dfrac{\ln(T)}{\ln\left(  \frac{q^{2}%
}{p^{2}}\right)  }+\dfrac{\gamma+\ln\left(  \frac{\chi_{\ell}(p)}{2\ell
}\right)  }{\ln\left(  \frac{q^{2}}{p^{2}}\right)  }+\dfrac{1}{2}%
+\varphi_{\ell}(T),
\]
\[ \mathbb{V}\left(  M_{T}\right)  \sim\dfrac{\pi^{2}%
}{6}\dfrac{1}{\ln\left(  \frac{q^{2}}{p^{2}}\right)  ^{2}}+\dfrac{1}{12}%
+\psi_{\ell}(T)
\]
as $T\rightarrow\infty$. The symbol $\gamma$ denotes Euler's constant
\cite{Fn-heu}; $\varphi_{\ell}$ and $\psi_{\ell}$ are periodic functions of
$\log_{q^{2}/p^{2}}(T)$ with period $1$ and of small amplitude; also%
\begin{align*}
\chi_{1}(p) & =\frac{p(q-p)^{2}}{q^{3}},
\\
\chi_{2}(p) & =\frac{\left[  1+(q-p)\theta\right]
  ^{2}(q-p)^{2}}{8q^{6}},
\\
\chi_{3}(p) & =\frac{\left[  u+(q-p)^{2}\theta+\sqrt{2}(q-p)\sqrt{v+u\theta
}\right]  ^{2}(q-p)^{2}}{48pq^{9}}%
\end{align*}
where%
\[%
\begin{array}
[c]{ccc}%
u=1-2p+6p^{2}-8p^{3}+4p^{4}, &  & v=1+6p^{2}-28p^{3}+54p^{4}-48p^{5}+16p^{6},
\end{array}
\]%
\[
\theta=\sqrt{1+4pq+16p^{2}q^{2}}.
\]
We give a brief justification of the case $\ell=2$ in the Appendix; elaborate
supporting algebraic details for $2\leq\ell\leq3$ are found in
\cite{fl18}
--- this is  actually  related to the Gumbel distribution function given
by {$\exp(-e^{-x})$}. Many applications are given in
Louchard and Prodinger \cite{momP}.

A numerical approach is necessary for\ $\ell\geq4$.  We readily see
that, like $\PP_\ell$, the transition matrix $P_\ell$ has a QBD
structure: if we group $\ell$ by $\ell$ the rows in $P_\ell$, we have 
\begin{equation}
   \label{e:Pl}
P_\ell  = \vligne{B & A_1 \\ A_{-1} & A_0 & A_1 \\  & A_{-1} & A_0 &
                                                                 \ddots
  \\  & & \ddots & \ddots}
\end{equation}
with blocks $A_{-1}$, $A_0$, $A_1$ and $B$ of size $\ell$ given by
\begin{align*}
 & \left[\begin{array}{c|c|c}
    A_{-1} & A_0 & A_1
  \end{array}
\right]
 = \\
 & \quad
\left[
\begin{array}{cccc|cccc|cccc}
  p_{-\ell} &   p_{-\ell+1} & \cdots  &  p_{-1} & p_0 & p_1 & \cdots  &  p_{\ell-1} & p_\ell \\
  & p_{-\ell} &   &   &  p_{-1} & p_0 & p_1 &   &  p_{\ell-1} & p_\ell \\
 & & \ddots &\ddots & \vdots & \ddots & \ddots & \ddots & \vdots & \ddots & \ddots\\
  &  &  &  p_{-\ell} &  p_{-\ell+1} & \ldots  &  p_{-1} & p_0 & p_1 & \ldots  &
 p_{\ell-1} & p_\ell 
\end{array}
\right]
\end{align*}
and
\begin{equation}
   \label{e:B}
B = A_0 + A_{-1} \vone \cdot \ve_1\tr
\end{equation}
where $\vone$ and $\ve_1$ are vectors of size $\ell$, all components
of $\vone$ are equal to 1 and $\ve_1\tr = \vligne{1 & 0 & \cdots & 0}$.
This means that $B$ is obtained by adding to the first column of
$A_0$ the probability mass on each row of $A_{-1}$.

This allows us to use basic features of QBDs, as presented in Latouche
and Ramaswami~\cite{lr99}, an early reference being
Neuts~\cite{neuts81}.  The theory is well established, and efficient
numerical procedures are readily available.  The paper is organised as
follows.  We give in the next section some background properties of
QBDs and we analyse the stationary distribution of the Markov chain
$\{Z_t\}$; in particular, we determine the decay rate of its
stationary distribution.  In Section \ref{s:sojourn}, we analyse the
expected sojourn times in clumps and obtain simple expressions for the
asymptotics of $\pi_k$ and $\mathbb{E}(C_k)$ in (\ref{e:MTa}).  We
discuss in Section \ref{s:natural} the effect of parsing the sequence
$\{S_t\}$ at epochs which are multiples of $2\ell$, and of working with
$\{Z_t\}$, instead of using $\{S_t\}$ itself.


\section{Stationary distribution of $P_\ell$}
\label{s:stationary}

We denote by $\vpi$ the stationary probability vector of $P_\ell$:
$\vpi\tr = \vpi\tr P_\ell$, $\vpi\tr\vone = 1$ and we partition $\vpi$
in a manner conformant with $P_\ell$, writing
\[
\vpi\tr = \vligne{\vpi_0\tr & \vpi_1\tr & \vpi_2\tr & \ldots}
\]
with 
\[
\vpi_k\tr = \vligne{\pi_{\ell k} &  \pi_{\ell k +1} & \ldots &
  \pi_{\ell (k+1)-1}}   \qquad \mbox{ for $k=0, 1, \ldots$}
\]
As $p<q$, the Markov chain is positive recurrent and 
\begin{equation}
   \label{e:pioR}
\vpi_k\tr = \vpi_0\tr R^k, \qquad k \geq 0,
\end{equation}
where $R$ is the unique non-negative matrix with
  eigenvalues in the open unit disk, solution of
\[
 R^2 A_{-1} + R (A_0-I)  +  A_1  = 0
\]
(\cite[Chapter 3]{neuts81}, \cite[Chapter 6]{lr99}).
In expanded form, (\ref{e:pioR}) may be written as
\begin{equation}
   \label{e:pioRbis}
\vligne{\pi_{\ell k} &  \pi_{\ell k +1}, \ldots ,
  \pi_{\ell (k+1)-1}} = \vligne{\pi_0 & \pi_1 & \ldots \pi_{\ell-1}}
R^k, \qquad k \geq 0.
\end{equation}
The boundary vector $\vpi_0$  is the unique solution of the
linear system
\begin{align}
   \label{e:pia}
\vpi_{0}\tr (B+ R A_{-1} -I) & = \vzero  \\
   \label{e:pib}
\vpi_{0}\tr  (I-R)^{-1} \vone & = 1.
\end{align}

Another important matrix is denoted as $G$ and is the unique stochastic solution
  of the equation
\begin{equation}
   \label{e:G}
 A_{-1} + (A_0-I)  G +  A_1  G^2 = 0.
\end{equation}
Its physical meaning is that 
$G_{ij}$, $1 \leq i,j \leq \ell$, is the
probability that, starting from state $n+i$, the Markov chain
visits $n-\ell +j$ before any other state with index $s \leq n$,
independently of $n$.

There exist efficient and numerically stable algorithms to compute $R$
in general (Bini
{\it et al.} \cite{blm05} and \cite{bmsv06}) and for all practical
purposes,  $R$ may be considered to be known, once $A_{-1}$, $A_0$ and $A_1$
are given.  In these particular circumstances, however, we may use the
Toeplitz structure of $P_\ell$ and obtain an explicit representation
of both matrices.  This is shown for $R$ in the next section and in
Section~\ref{s:sojourn} for $G$.

\section{Toeplitz structure}
   \label{s:toeplitz}

   Our analysis of $R$ and $G$ is based primarily on Bini {\it et
     al.}~\cite[Chapter 5]{blm05}, which deals with M/G/1-type Markov
   chains with limited displacements such as $P_\ell$ in
   (\ref{e:Pl}). This approach leads us to determine explicitly the
   eigenvalues and eigenvectors of these two matrices.  We need, in
   this section and the next, to define the random walk on
   $(-\infty, +\infty)$ with jump size distribution
   $\{p_n: -\ell \leq n \leq \ell\}$; it is a Markov chain with
   transition matrix
\begin{equation}
   \label{e:ptilde}
\widetilde P_\ell = 
\vligne{\ddots & \ddots \\ 
\ddots  & A_0 & A_1 \\
& A_{-1} & A_0 & A_1 \\ & & A_{-1} & A_0  & \ddots \\
  &  &  & \ddots  & \ddots
}.
\end{equation}

The next lemma is stated without proof, details are in \cite[Section
4.4]{blm05} and \cite[Chapter 9]{lr99}.
\begin{lem}
   \label{t:roots}
The  roots of the polynomial  $\xi(z)= \det \Xi(z)$, with
\[
\Xi(z) = A_{-1} + z (A_0-I) + z^2 A_1,
\] 
are
\[
|z_1| \leq |z_2| \leq \cdots \leq |z_{\ell -1}| < z_\ell = 1 <
z_{\ell+1} < 
|z_{\ell+2}| \leq \cdots \leq |z_{2\ell}|.
\]
The roots $z_1$ to $z_\ell$ are the eigenvalues of $G$.
The roots $z_{\ell+1}$, $z_{\ell+2}$, \ldots, $z_{2\ell}$ are the reciprocals of the
eigenvalues of $R$
 and $\eta$, the
Perron-Frobenius eigenvalue of  $R$, is equal to $1/z_{\ell+1}$.   
\qed
\end{lem}

\begin{rem} \em
   \label{r:eta}
   The polynomial $\xi(z)$ has degree $2\ell$ since the matrix $A_1$
   is nonsingular. Moreover, the transition matrix $\widetilde P$ is
   irreducible, and so the quantity $\kappa$ defined in \cite[Eqn
   (4.26)]{blm05} is equal to 1.  Therefore, it results from
   \cite[Theorem 4.9]{blm05} that 
the polynomial $\xi(z)$ does not have roots of modulus $z_\ell$ and $z_{\ell+1}$, different from $z_\ell$ and $z_{\ell+1}$, respectively.

In consequence, $\eta$ has
  multiplicity 1, all other  eigenvalues of $R$ have a modulus
  strictly less than  $\eta$, and  it
immediately results from (\ref{e:pioR}) that
\begin{equation}
   \label{e:piklim}
\vpi_k\tr = (\vpi_0 \tr \vv) \vu\tr \eta^k  + o(\eta^k)
\qquad \mbox{asymptotically as $k \rightarrow \infty$},
\end{equation}
where $\vu$ and $\vv$ are respectively the left- and right-eigenvector
of $R$ for the eigenvalue $\eta$, normalised so that
$\vu\tr \vv = 1$.
\end{rem}


To begin with, we define 
\[
p(z)= \sum_{-\ell \leq i \leq \ell}   p_i z^i,
\]
with the intention of using \cite[Theorem 5.17]{blm05}: if $\zeta$ is
a root of $z^\ell - z^\ell  p(z)$, then $\zeta^\ell$ is a root of
$\xi(z)$.

\begin{lem}
   \label{t:pofz}
The function $p(z)$ is given by $p(z)= (p+q z^{-1})^\ell (q + pz)^\ell$.
\end{lem}
\begin{proof}
  By using the isomorphism between semi-infinite Toeplitz matrices and
  Laurent power series (see \cite[Section 3.1]{blm05}), the Toeplitz
  part of the matrix $U$ is represented by $u(z)= q+pz$ while the
  Toeplitz part of $V$ is represented by $v(z)= p+q z^{-1}$, and so
  the representation $p(z)$ of the Toeplitz part of
  $P_\ell = U^\ell V^\ell$ is given by $u(z)^\ell v(z)^\ell$.
\end{proof}

\begin{lem}
   \label{t:rootsq}
Define $\omega_\ell = \cos(2\pi/\ell) + \ic \sin(2\pi/\ell)$, where
$\ic$ is a complex square root of $-1$.  The roots of the
polynomial $h(z) = z^\ell - z^\ell  p(z)$ are the roots of the
quadratic polynomials $s_j(z) = (pz+q)^2 - \omega_\ell^{j-1} z$, for
$1 \leq j \leq \ell$.  

Denoting by $\sigma_j$ and $\mu_j$
the two roots of $s_j(z)$, with $|\sigma_j| \leq |\mu_j|$, we have
$|\sigma_j| \leq 1 < |\mu_j|$.  In particular, $\sigma_1 = 1$ and
$\mu_1 = (q/p)^2$.

The eigenvalues of $G$ are $\{\sigma_j^\ell: 1\leq j \leq \ell\}$; the
eigenvalues of $R$ are the reciprocals of $\{\mu_j^\ell: 1\leq j \leq \ell\}$.
\end{lem}

\begin{proof}
By Lemma \ref{t:pofz}, we have 
\begin{equation}
   \label{e:hofz}
h(z) = z^\ell -(q+pz)^{2 \ell},
\end{equation}
and
so the roots of $h(z)$ are such that $((q+pz)^2/z)^\ell=1$, or
$((q+pz)^2/z)=\omega_\ell^{j-1}$ for some $j$.  This proves the first
statement.

Take $\epsilon \in (0,1)$ and $f_{1,\epsilon}(z)= (q-\epsilon + pz)^2$
and define $f_2(z)= \omega_\ell^{j-1} z$.  By Rouch{\'e}'s theorem, as
$|f_{1, \epsilon}(z)| \leq (q+p-\epsilon)^2 < |f_2(z)| = 1$ for
$|z|=1$, the polynomial
$s_{j,\epsilon}(z) = f_{1, \epsilon}(z) - f_2(z)$ has one root inside
and one root outside the unit circle.  Taking the limit as
$\epsilon \rightarrow 0$, we obtain $|\sigma_j| \leq 1 \leq |\mu_j|$.
Since $\sigma_j . \mu_j = (q/p)^2 >1$, we conclude that $|\mu_j|> 1$
and this proves the second statement.  It is easily verified that
$\sigma_1 = 1$ and $\mu_1 = (q/p)^2$.

The last statement is a consequence of \cite[Theorem 5.17]{blm05}, as
we mentioned immediately before Lemma \ref{t:pofz}.
\end{proof}

Given a vector $\vx = \vligne{x_0 & \ldots & x_{\ell-1}}\tr$, one
defines the companion matrix associated to $\vx$ as
\[
C(\vx) = \vligne{
0 & 1 & 0 & \ldots & 0 \\
\vdots & \ddots & \ddots & \ddots & \vdots \\
\vdots &  & \ddots & \ddots & 0 \\
0 & \ldots & \ldots & 0 & 1 \\
x_0 & x_1 & \ldots & \ldots & x_{\ell-1}
};
\]
it is non-singular if and only if $x_0 \not= 0$ and its eigenvalues are the roots
of the polynomial $z^\ell - \sum_{0 \leq i \leq \ell-1} x_i z^i$.  

The following theorem is given without proof: it results from Lemma
\ref{t:rootsq} and from \cite[Section 5.5]{blm05}.

\begin{lem}
   \label{t:R}
Define the polynomial 
\begin{equation}
   \label{e:rofz}
r(z)= z^\ell - \sum_{0 \leq i \leq \ell-1} r_i
z^i = \prod_{1 \leq i \leq \ell} (z-\mu_i^{-1})
\end{equation}
and the vector $\vr = \vligne{r_0 & \ldots & r_{\ell-1}}\tr$.  The
matrix $R$ is given by 
\begin{equation}
   \label{e:rcr}
R= (C(\vr)^\ell)\tr,
\end{equation}
it is nonsingular since all $\mu_i^{-1}$s are different from 0 and
it has the LU-type decomposition $R=L_R U_R ^{-1}$ with
\begin{equation}
   \label{e:rcrbis}
L_R =
\vligne{r_0 \\
r_1 & r_0 \\
\vdots & \ddots & \ddots \\
r_{\ell-1} & \ldots & r_1 & r_0
},
\qquad 
U_R=
\vligne{
 1 & -r_{\ell-1} & \ldots & -r_1 \\
  & \ddots &  \ddots & \vdots  \\
  &  & 1  & -r_{\ell-1} \\
  &  &  & 1  \\
}
\end{equation}
\qed
\end{lem}

The spectral decomposition of $R$ given below is a straightforward
consequence of Lemma \ref{t:R}.

\begin{thm}
   \label{t:Rsp}
The matrix $R$ is given by
\begin{equation}
   \label{e:Rsp}
R = V_R^{-1} D_R^\ell V_R
\end{equation}
where $V_R$ is the Vandermonde matrix
\[
V_R = \vligne{
1 & \mu_1^{-1} & \ldots & \mu_1^{-(\ell-1)} \\
1 & \mu_2^{-1} & \ldots & \mu_2^{-(\ell-1)} \\
\vdots \\
1 & \mu_\ell^{-1} & \ldots & \mu_\ell^{-(\ell-1)} \\
}
\]
and $D_R$ is the diagonal matrix
\[
D_R = \vligne{ 
\mu_1^{-1} \\
 & \mu_2^{-1} \\
 & & \ddots \\
 & & & \mu_\ell^{-1}
}.
\]
All left eigenvectors of $R$ have the form $\vligne{1 & \mu_j^{-1} & \ldots & \mu_j^{-(\ell-1)} }$.
\end{thm}

\begin{proof}
  The matrix $R$ is diagonalisable because the $\mu_j$ are all
  different, and by (\ref{e:rcr}), we only need to verify that
  $C(\vr)\tr = V_R^{-1} D_R V_R$, or $V_R C(\vr)\tr = D_R V_R$.  For
  the $j$th row of $V_R$, we write
\begin{align*}
\vligne{1 & \mu_j^{-1} & \ldots & \mu_j^{-(\ell-1)} } C(\vr)\tr
 & = \vligne{\mu_j^{-1} & \ldots & \mu_j^{-(\ell-1)} & \sum_{0 \leq
      \nu \leq \ell-1} r_\nu \mu_j^{-\nu}} \\
 & = \vligne{\mu_j^{-1} & \ldots & \mu_j^{-(\ell-1)} & \mu_j^{-\ell}}
\end{align*}
by the definition (\ref{e:rofz}) of the polynomial $r(z)$.
\end{proof}

\begin{cor}
   \label{t:asymptotic}
The maximum eigenvalue of $R$ is $\eta = \rho^{2 \ell}$, and its 
corresponding left eigenvector is 
\begin{equation}
   \label{e:u}
\vu = \vligne{1 & \rho^2 & \ldots & \rho^{2(\ell-1)}}\tr,
\end{equation}
where $\rho = p/q$.  
\end{cor}
\begin{proof}
The components of the eigenvector
  $\vligne{1 & \mu_1^{-1} & \ldots & \mu_1^{-(\ell-1)}}$ are all
  strictly positive since $\mu_1^{-1} = (p/q)^2 > 0$, and 
 by the Perron-Frobenius theorem it is associated to the maximal eigenvalue.
\end{proof}

We now turn to the equations (\ref{e:pia}, \ref{e:pib}) that
characterise the vector $\vpi_0$.

\begin{thm}
   \label{t:pio}
The vector $\vpi_0$ is given by
\begin{equation}
   \label{e:pio}
\vpi_0\tr  = c \ve_1\tr R,
\end{equation}
where
$
c = -\prod_{1 \leq \nu \leq \ell } (1-\mu_\nu)
$
and $\ve_1 = \vligne{1 & 0 & \ldots & 0}\tr$.
\end{thm}
\begin{proof}
First, we replace $B$ in (\ref{e:pia}) by its expression (\ref{e:B})
and so
\[
\vpi_0\tr (I - A_0 - R A_{-1}) = (\vpi_0\tr A_{-1} \vone) \ve_1\tr.
\]
As $\ve_1\tr A_1 = p_\ell \ve_1\tr$, we rewrite the last equation
and obtain
\[
\vpi_0\tr = p_\ell^{-1} (\vpi_0\tr A_{-1} \vone) \ve_1\tr A_1 (I - A_0
- R A_{-1})^{-1} = p_\ell^{-1} (\vpi_0\tr A_{-1} \vone) \ve_1\tr R
\]
by \cite[Eqn (6.9)]{lr99}, which proves (\ref{e:pio}) for some scalar
$c$ to be determined.

We replace in (\ref{e:pib}) $\vpi_0$ by its expression in
(\ref{e:pio}) and $R$ by its expression $R=L_R U_R ^{-1}$ from Lemma
\ref{t:R}, and we obtain after simple manipulations that
\begin{equation}
   \label{e:pioA}
c \, r_0 \ve_1\tr (U_R - L_R)^{-1} \vone = 1.
\end{equation}

The matrix $U_R - L_R$ is a circulant matrix and may be written as
\[
U_R - L_R = \overline{S} \diag(\vr^*) S,
\]
with 
\[
S = \frac{1}{\sqrt{\ell}}  \Omega_\ell,
\qquad
\vr^* = \Omega_\ell \vligne{1-r_0 \\ -r_1 \\ \ddots \\ -r_{\ell-1}}
\]
and $(\Omega_\ell)_{j,j'} = \omega_\ell^{jj'}$, $0 \leq j, j' \leq
\ell-1$ (see \cite[Theorem 2.4]{blm05}).   
Thus, (\ref{e:pioA}) may be written as
\begin{equation}
   \label{e:pioB}
c \, r_0 \ve_1\tr \overline{S} \diag(\vr^*)^{-1} S \vone = 1.
\end{equation}
One easily verifies that $\ve_1\tr \overline{S} =
\frac{1}{\sqrt{\ell}} \vone\tr$ and $S \vone = \sqrt{\ell} \ve_1$, and
(\ref{e:pioB}) becomes 
\[
c \, r_0 \vone\tr \diag(\vr^*)^{-1} \ve_1 = 1
\qquad 
\mbox{or}
\qquad
c = r_0^* r_0^{-1}.
\]
With $r_0 = -r(0) = -\prod_{1 \leq i \leq \ell} (-\mu_i^{-1})$ and
$r_0^* = r(1) = \prod_{0 \leq i \leq \ell-1} (1-\mu_i^{-1})$, this
completes the proof.
\end{proof}

We bring together the different results in this section, and rewrite
(\ref{e:piklim}) as $\vpi_k\tr \sim (\vpi_0\tr \vv) \vu\tr \eta^k$ or,
if we focus on the first components of $\vpi_k$ and $\vu$,
\[
\pi_{n} \sim c(\ell, p) \rho^{2 n} \qquad \mbox{as $n
  \rightarrow \infty$},
\]
where 
\begin{align}
  \nonumber
c(\ell, p) &= \vpi_0\tr \vv = c \ve_1\tr R \vv = -\prod_{1 \leq \nu
   \leq \ell} (1-\mu_\nu) \rho^{2 \ell} \ve_1\tr V_R^{-1} \ve_1 
  \qquad \mbox{by (\ref{e:Rsp})}
\\
  \nonumber
& = -\prod_{1 \leq \nu
   \leq \ell} (1-\mu_\nu) \rho^{2 \ell} (V_R^{-1})_{11}
\\
   \nonumber
& = -\prod_{1 \leq \nu \leq \ell} (1-\mu_\nu) \rho^{2 \ell}
  \prod_{2 \leq \nu \leq \ell} \mu_\nu^{-1} (\mu_\nu^{-1} -\mu_1^{-1})^{-1}
\\
  \nonumber
  & = -(1-\mu_1) \rho^{2 \ell} \prod_{2 \leq \nu \leq \ell} (1-\mu_\nu)
    (1-\mu_1^{-1} \mu_\nu)^{-1}
\\
   \label{e:clp}
  & = (\mu_1-1) \mu_1^{-1} \prod_{2 \leq \nu \leq \ell} (1-\mu_\nu)
    (\mu_1 - \mu_\nu)^{-1} 
\end{align}
by Lemma \ref{t:rootsq} --- we refer to Macon and Spitzbart \cite{ms58}
to justify the third
  equality.

\section{Sojourn times}
\label{s:sojourn}

Briefly stated, to follow the Poisson clumping approach in
Aldous~\cite{Ald-heu}, one firstly reformulates questions about the
distribution of $M_T$ as questions about random sets.  To this effect,
we need to define a clump $E_k$ for each state $k$, that is, a set of
states such that
\[
\mbox{the event} \ 
[M_T < k]  \ \mbox{is equivalent to} \   \{{\cal T}_k \cap [0,T]\} \ \mbox{being empty},
\]
where ${\cal T}_k = \{t: S_t \in E_k \}$.  As the transition matrix
$P_\ell$ of $\{Z_t\}$ allows for jumps, we need a clump big enough
that the process is forced to visit one of its states whenever
$M_T \geq k$, for this reason, we chose
$E_k = \{k, k+1, \ldots, k+\ell-1\}$.

{ We approximate each random set ${\cal T}_k$ by a mosaic; in
  the present case, the mosaic is a sparse Poisson process of
  independent random sets on the real line, each set being distributed
  as intervals of continuous sojourn in $E_k$.  }

{Next, we take the
  stationary version of the process $\{X(t)\}$; if $k$ is large
  enough, the random visits to $k$ are close to a Poisson process and,
  repeating Aldous~\cite[Section B8]{Ald-heu}, one finds that the
  clump rate is $\pi_k / \E(C)$   where ${\mathbb{E}(C)}$ is the sojourn
 time in a clump.
This leads to the
  approximation (\ref{e:MTa}).
}

Finally, 
following \cite[Section B12]{Ald-heu}, we approximate the
expected sojourn time in the clump by the total expected sojourn time
in $E_k$ for the random walk on $(-\infty, +\infty)$ with jump size
distribution $\{p_n: -\ell \leq n \leq \ell\}$, that is, the Markov
chain with transition matrix (\ref{e:ptilde}).

As $p < q$, this Markov chain is transient and drifts to $- \infty$, so that
the total expected number of visits $w_{ij}$ to state $j$,
starting from state $i$, during the whole history of the process is
finite for all $i$ and $j$.  We denote by $W$  the matrix with entries
$W_{i,j} = w_{\nu+i,\nu+j}$ for $0 \leq i,  j \leq \ell-1$, {\em independently}
of $\nu$.


To determine $W$, we need two matrices: $G$ defined in (\ref{e:G}) and
$H$, similar to $G$ in the reverse direction, that is, for
$0 \leq i, j \leq \ell-1$, $H_{ij}$ is the probability that, starting
from state $\ell n+i$, the Markov chain visits $\ell (n+1)+j$ before
any other state with index $s >\ell n + \ell-1$, independently of
$n$.  The matrix $H$ 
is the unique sub-stochastic solution of the equation
\begin{equation}
   \label{e:H}
A_{-1} H^2 + (A_0 - I) H + A_1 =0.
\end{equation}
One easily verifies that
\begin{align}
  \nonumber
W & = I + (A_0 + A_1 G + A_{-1} H) + (A_0 + A_1 G + A_{-1} H)^2 + \cdots
  \\
  \nonumber
 & = \sum_{\nu \geq 0} (A_0 + A_1 G + A_{-1} H)^\nu \\
   \label{e:W}
 & = (I - (A_0 + A_1 G + A_{-1} H))^{-1}.
\end{align}
The next lemma  provides us with the justification for (\ref{e:MT}).
\begin{lem}
   \label{t:lambda}
For large values of $k$,
\begin{equation}
   \label{e:explambda}
\P[M_T < k]  \approx \exp(- \frac{\chi_\ell(p)}{2 \ell}T \rho^{2k})
\end{equation}
with 
\[
\chi_\ell(p)= c(\ell,p) \, \vu\tr (I - (A_0 + A_1 G + A_{-1} H)) \vone.
\]
\end{lem}
\begin{proof}
The proof immediately follows from \cite[Eq. B12a]{Ald-heu}: the first
passage time $\tau_k$ to $E_k$ is approximately exponential with
parameter $\lambda_k= \vlambda_k\tr \vone$, where $\vlambda_k$ is the solution to the
linear system
$
\vlambda_k\tr W = \vy\tr 
$
with $\vy\tr = \vligne{\pi_k & \pi_{k+1} & \cdots & \pi_{k+\ell
    -1}}$.
The factor $1/(2\ell)$ in the right-hand side of (\ref{e:explambda})
is required because $\lambda_k$ is determined by the Markov chain
(\ref{e:Pl}) and each unit of time there represents a full cycle
of size $2 \ell$ of the traffic light. 

By (\ref{e:W}), 
\[
\lambda_k\tr = \vy\tr (I - (A_0 + A_1 G + A_{-1} H)) \vone
\]
and by (\ref{e:piklim}, \ref{e:u}), 
\[
\vligne{\pi_k & \ldots & \pi_{k+\ell-1}} =
c(\ell,p) \rho^{2k} \vu\tr + o(\rho^{2k}),
\]
so that
\[
\P[M_T < k]  \approx \exp(- c(\ell,p) \vu\tr (I - (A_0 + A_1 G + A_{-1}
H)) \vone \rho^{2k}\frac{T}{2 \ell})
\]
which we rewrite as (\ref{e:explambda}).
\end{proof}

Like we did in Section \ref{s:toeplitz} for the matrix $R$, we may use
the spectral decomposition approach to obtain explicit expressions for
$G$ and $H$.  This, however, has not helped us obtain a more explicit
expression for $\chi_\ell(p)$ and we relegate the details to Appendix
\ref{s:GH}.

We illustrate on Figures \ref{f:mt1}, \ref{f:mt2} and \ref{f:mt3} the
quality of the approximation given in Lemma~\ref{t:lambda} for the
distribution of $M_T$.  The red bars are the analytical approximation
(\ref{e:explambda}); the blue density is obtained by simulation of the
random walk (\ref{e:Pl}).  The number of replications of the
simulation is 40,000 and $\ell = 4$ in all cases; $T=10^4$ in Figure
\ref{f:mt1} and $T= 10^9$ in Figures \ref{f:mt2} and \ref{f:mt3}.  The
probability $p$ that a car arrives is $0.40$ in Figures \ref{f:mt1}
and \ref{f:mt2}, and $0.45$ in Figure \ref{f:mt3}.

We have mentioned earlier that $k$ has to be sufficiently large for
(\ref{e:MTa}) to hold.  A first observation from Figures \ref{f:mt2}
and \ref{f:mt3} is that the quality of the approximation is striking
over the whole range of values of $k$ when $T= 10^9$.  It is very good
even for the much smaller $T=10^4$, as we see on Figures \ref{f:mt1}
and \ref{f:mt2} for both of which $p=0.40$.  We have run other
simulations, letting $T$ range from $10^5$ to $10^8$, and have reached
similar conclusions.  We have also increased the number
of replications,  and this does not
alter the graphs nor parameter estimates significantly.

We observe from the figures, as expected, that the range of values for
$M_T$ increases if $T$ becomes bigger, for a given $p$ (more time for
the queue to reach high values), or if $p$ increases, for a given $T$
(more chances for the queue to grow during intervals when the traffic
light is red).

\begin{rem} \em
   \label{r:conject}
We conjecture that $\chi_\ell(p) / (2 \ell) = q^2 c^2(\ell,p)$.  This
identity is proved formally in Appendix~\ref{s:exact2} for $\ell =2$ and in
\cite{fl18} for $\ell=3$ and we have experimental evidence from our
numerical analysis, the computed difference being of the order of
$10^{-15}$.   However, we have not been able to show this
analytically in all generality.
\end{rem}

\section{Detailed queue representation}
\label{s:natural}

We apply in this section the QBD theory directly to the original transition matrix $\PP_\ell$ of
(\ref{e:Porig}).  Its stationary probability vector $\vgamma$ is
partitioned into sub-vectors $\vgamma_0$, $\vgamma_1$, $\vgamma_2$, \ldots, with
\[
\vgamma_k = \vligne{\gamma_{k,1} & \ldots & \gamma_{k,2\ell}}
\]
where the first index is the number of cars waiting in front of the
traffic light and the second is the position of the traffic light in
the cycle $R\cdots RG\cdots G$ of length $2 \ell$.  It is given by
\begin{equation}
   \label{e:gamma}
\vgamma_k = \vgamma_0 \whR^k, \qquad k \geq 0,
\end{equation}
where $\whR$ is the non-negative solution of $\whR^2 \ca_{-1} + \whR
(\ca_0 - I) + \ca_1 = 0$.  We denote by $\wheta$ its maximal
eigenvalue and by $\whu$ and
$\whv$ the corresponding left- and right-eigenvectors.  Finally, we
denote by $\whpi_k = \vgamma_k\tr \vone$ the stationary marginal
probability that $k$ cars are waiting.

\begin{thm}
   \label{t:asymptorig}
The maximal eigenvalue of $\whR$ is $\wheta = \rho^2$.  The
corresponding left eigenvector is
\begin{equation}
   \label{e:whu}
\whu = \vligne{\rho^\ell & \rho^{\ell - 1} & \cdots & \rho & 1 & \rho
  & \cdots & \rho^{\ell-1}}\tr  \\
\end{equation}
and
\begin{equation}
   \label{e:whpi}
\whpi_k = \whc(\ell,p) \rho^{2 k} + o(\rho^{2k})
\end{equation}
asymptotically as $k \rightarrow \infty$, with 
\begin{equation}
   \label{e:waofp}
\whc(\ell,p) = \frac{1-\rho^2}{2 \ell} (\whv\tr \vone)   (\whu\tr \vone).
\end{equation}
\end{thm}
\begin{proof}
By Lemma~\ref{t:roots} applied to (\ref{e:Porig}), $\wheta$ is the
largest root strictly less than 1 of the polynomial $\det \ca(z)$,
with 
\begin{align*}
\ca(z) & = z^2 \ca_{-1} + z (\ca_0 -I) + \ca_1  
\\
& = \vligne{
-z & p+zq \\
 & \ddots & \ddots \\
 & & -z & p+zq \\
 & &  & -z & z (p+zq) \\
 & & &  & \ddots & \ddots \\
z(p+zq) & & & & & -z
}.
\end{align*}
Simple calculations show that $\det \ca(z)$ is a polynomial of degree
$3 \ell$ given by
\[
\det \ca(z) = z^\ell(z^\ell - (p+zq)^{2\ell})),
\]
so that its roots are 
\begin{itemize}
\item 
0 with multiplicity $\ell$ and 
\item 
the inverse of
the $2 \ell$ roots of $h(z)$ in (\ref{e:hofz}).  
\end{itemize}
Therefore, $\whR$ has
$\ell$ eigenvalues equal to 0 and $\ell$ eigenvalues given by
$1/\mu_j$, $1 \leq j \leq \ell$; furthermore, 
$\wheta = 1/\mu_1  = \rho^2$, it has multiplicity 1 and all other
eigenvalues are strictly less than $\wheta$ in modulus.  One easily
verifies that $\whu\tr
\ca(\rho^2) = \vzero$, so that $\whu$ is 
the Perron-Frobenius eigenvector of $\whR$.

It results from (\ref{e:gamma}) that
\[
\vgamma_k = (\vgamma_0\tr \whv) \whu \ \rho^{2 k} + o(\rho^{2 k} )
\] 
and from \cite[Lemma 6.3.2]{lr99} that $\vgamma_0\tr = \valpha\tr
(I-R)$ where $\valpha$ is the stationary probability vector of the
matrix 
$\ca = \ca_{-1} + \ca_0 + \ca_1$.  As $\ca$ is doubly stochastic,
$\valpha = 1/(2 \ell) \vone$, and so
\begin{align*}
\whpi_k & = \vgamma_k\tr \vone \\
   & = 1/(2 \ell) \vone\tr (I-R) \whv \ \whu\tr \vone \rho^{2k} +
     o(\rho^{2k})  \\
   & = (1-\rho^2)/(2 \ell) \vone\tr \whv \ \whu\tr \vone \rho^{2k} +
     o(\rho^{2k})
\end{align*}
and this concludes the proof.
\end{proof}

We show on Figure \ref{f:pipiHp40l10} the distributions of the
detailed queue $\{S_t\}$ and its approximation $\{Z_t\}$.  We
clearly see the effect of the difference  of behaviour when the queue is
nearly empty.  Nevertheless, comparing (\ref{e:whpi}) with
(\ref{e:piklim}), we see that the 
distribution of the number of waiting cars has the same asymptotic
decay.

\begin{figure}[tb]
\centering
\includegraphics[scale=0.5]{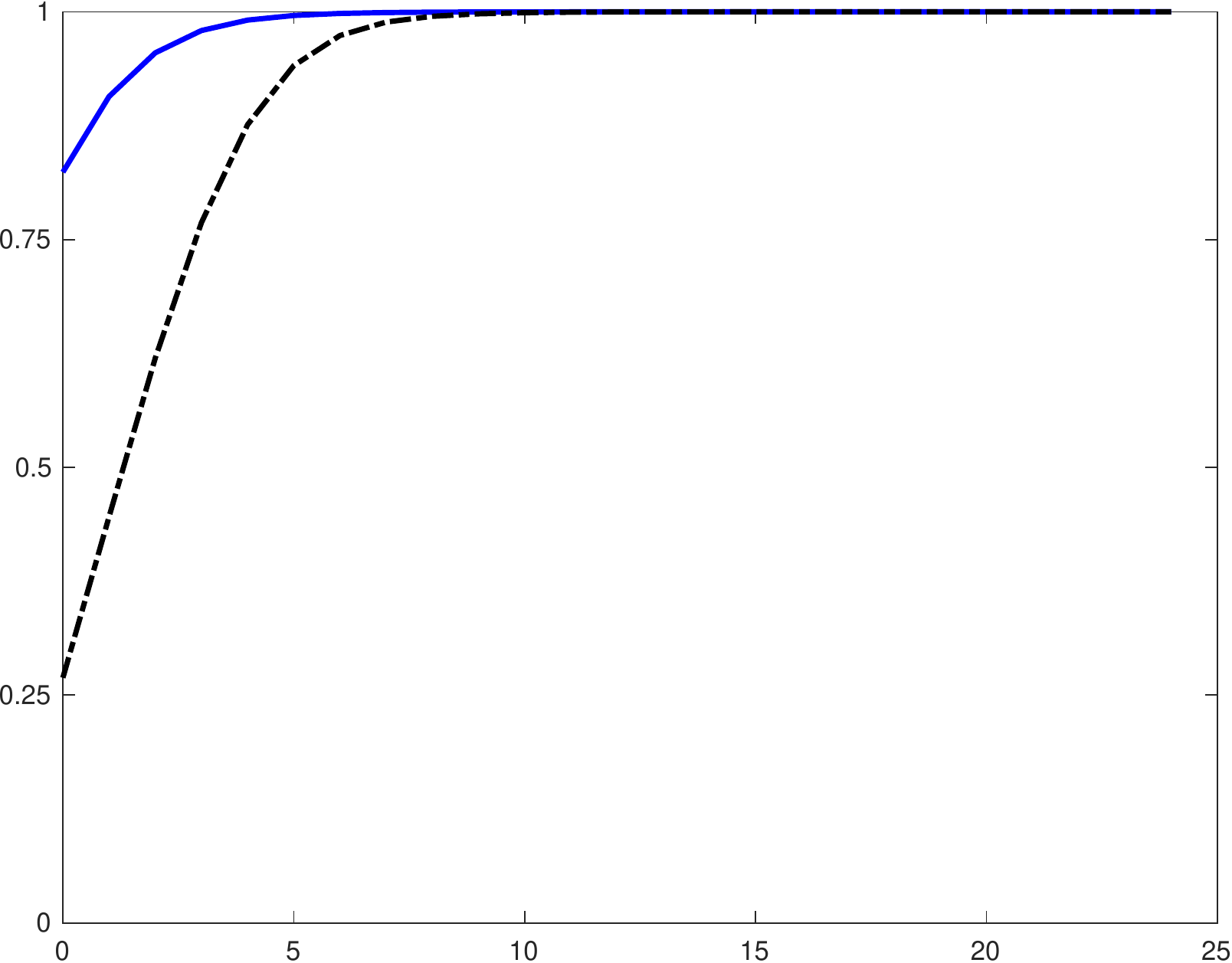} 
 \caption{
  \label{f:pipiHp40l10}
Cumulative probability distribution for the two models, the continuous
blue line for the random
walk model, the dashed black line for the detailed model.  The
parameters are $\ell=10$, $p=0.4$.}
\end{figure}

\begin{figure}[tb]
\centering
\includegraphics[scale=0.5]{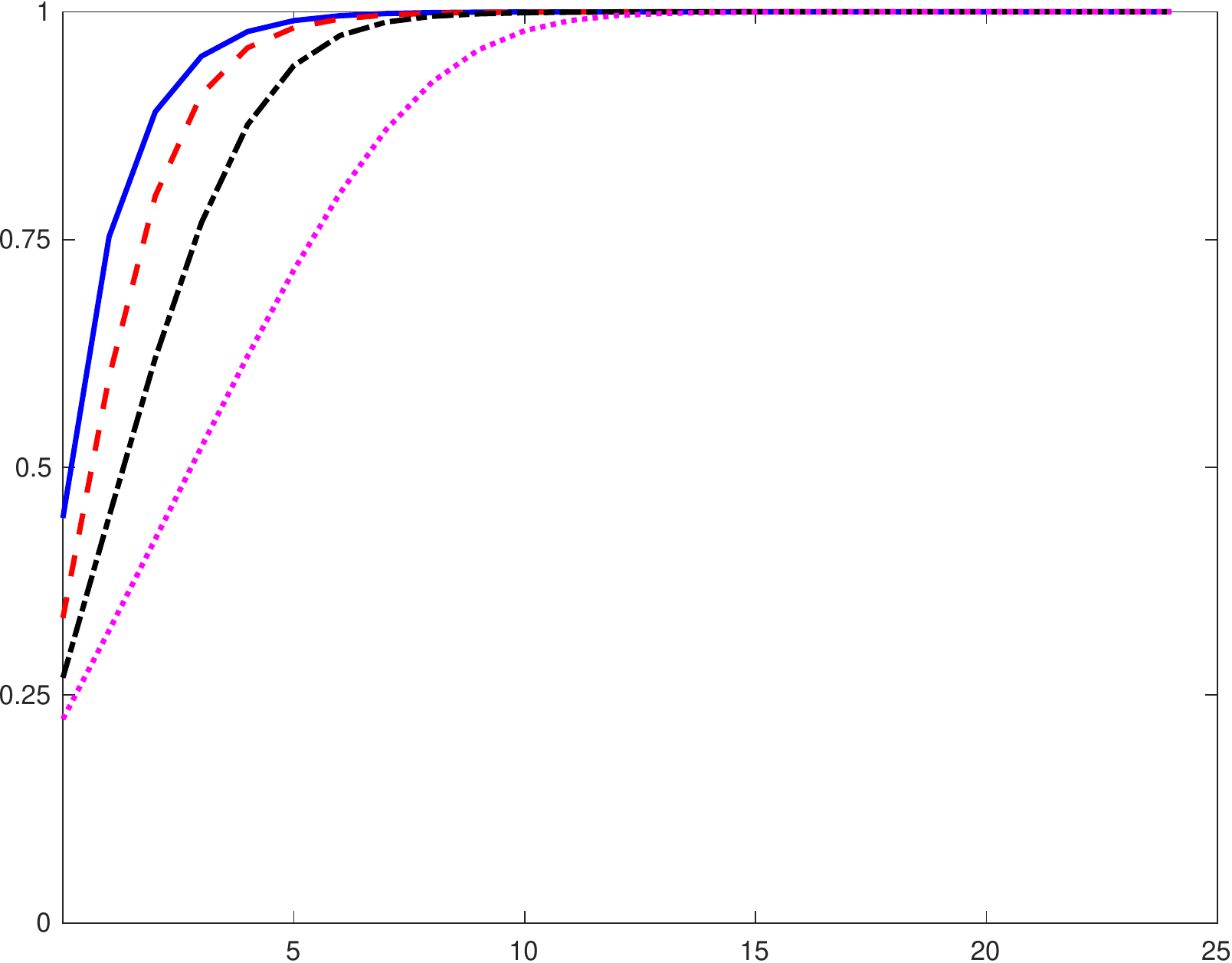} 
 \caption{
  \label{f:piH40l151020}
Cumulative probability distribution for the detailed model, for
$p=0.4$.  The continuous  blue
line corresponds to $\ell = 1$, the red dashed line to
$\ell=5$, the black dot-dashed line to $\ell = 10$ and the magenta dotted
line is for $\ell= 20$.}
\end{figure}

We give on Figure \ref{f:piH40l151020} the distribution $\whpi$ for
different values of $\ell$.  One observes that the queue becomes
stochastically much greater as $\ell$ increases.  This is explained by
the fact that the queue builds up during an interval of red light.
Under normal circumstances, it will be served during the succeeding
interval of green light, but it is possible that some cars remain when
the light becomes red again, so that build-ups might accumulate for a
while --- despite of which, the {\em decay} of the distribution
remains the same $\rho^2$ independently of $\ell$.

Finally, to determine the tail of the distribution of $\whM_T$, we use
for clumps 
the sets $\whE_k = \{(k,i): 1 \leq i \leq 2 \ell\}$.  The proof of
Lemma \ref{t:whM} is similar to that of Lemma \ref{t:lambda} and 
is omitted.


\begin{lem}
   \label{t:whM}
For large values of $k$,
\begin{equation}
   \label{e:whM}
\P[\whM_T < k]  \approx \exp(- \frac{\whchi_\ell(p)}{2 \ell} T \rho^{2k})
\end{equation}
with 
\[
\whchi_\ell(p) = (1-\rho^2) (\whv\tr \vone)   (\whu\tr
(I-(\ca_0 + \ca_1\whG + \ca_{-1} \whH))\vone) 
\]
where $\whG$ and $\whH$ are respectively the solutions of
\begin{align*}
\ca_{-1} + (\ca_0 -I) \whG + \ca_1 \whG^2  = 0 
\qquad \mbox{and}  \qquad 
\ca_{1} + (\ca_0 -I) \whH + \ca_{-1} \whH^2 & = 0. 
\qquad \mbox{\qed}
\end{align*}
%
\end{lem}

\begin{figure}[tb]
\centering
\includegraphics[scale=0.5]{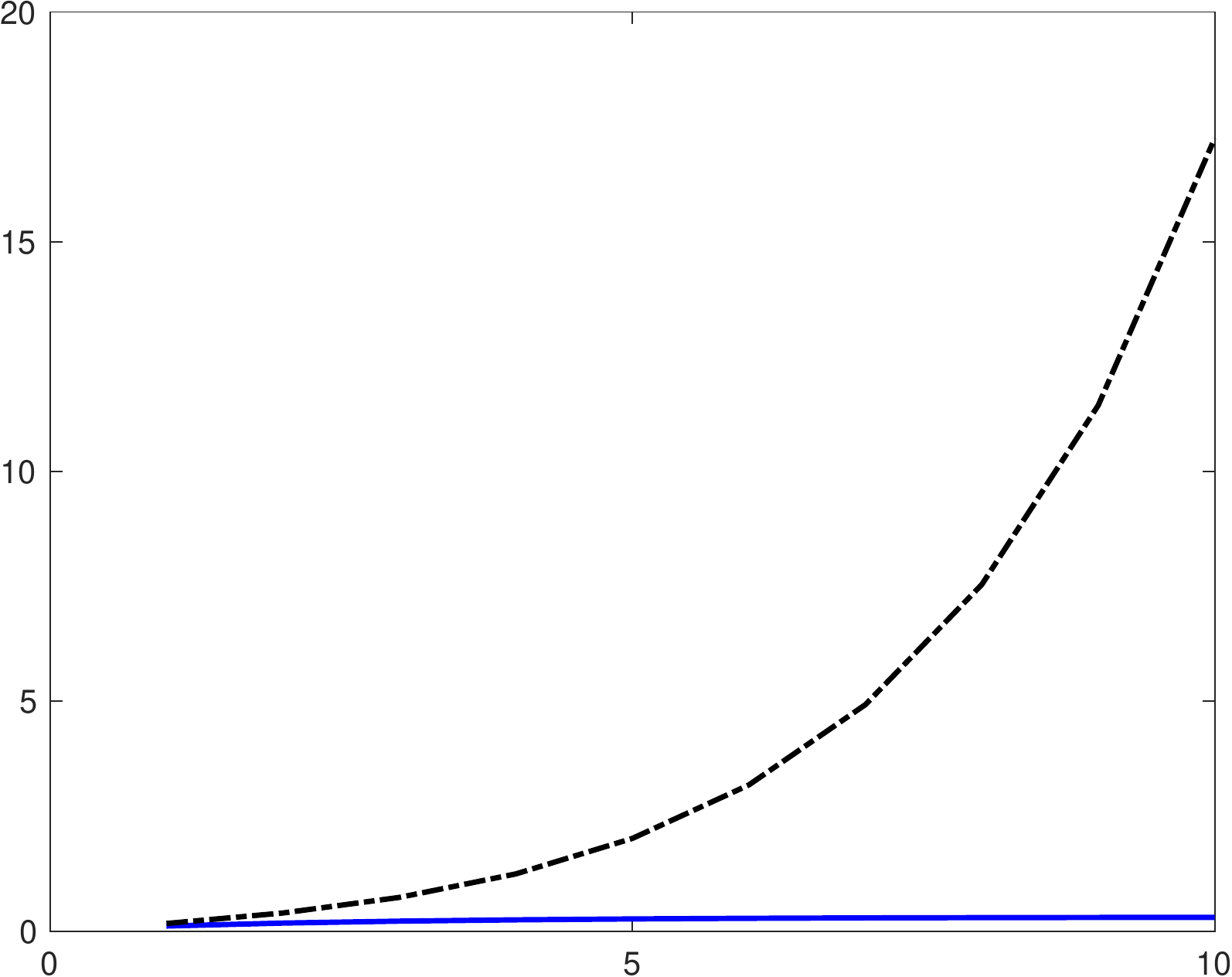} 
 \caption{
  \label{f:lambda40}
Values of $\chi_\ell(p)$ (continuous blue line), and
$\whchi_\ell(p)$ (dashed black line), for $p=0.4$ and $\ell
= 1$ to 10.}
\end{figure}

A comparison of $\chi_\ell(p)$ and $\whchi_\ell(p)$ from
Lemmas~\ref{t:lambda} and \ref{t:whM} show that they are very
different.  As an illustration, we give their values on Figure
\ref{f:lambda40} for $p=0.4$ and $\ell = 1$ to 10.
 However, we have noticed that there appear to be a
  systematic ratio between the two: in all our numerical investigation we
  have observed that 
\[
| \frac{\chi_\ell(p) }{ \whchi_\ell(p) \rho^\ell} - 1| < 2 \cdot 10^{-15}.
\]
This leads us to conjecture that  
$\chi_\ell(p)= \whchi_\ell(p) \rho^\ell$.  Under this conjecture, we
re-write (\ref{e:explambda}) as
\[
\P[M_T < k]  \approx \exp(- \frac{\whchi_\ell(p)}{2 \ell}T \rho^{2k+ \ell})
\]
which, together with (\ref{e:whM}), indicates that 
\begin{equation}
   \label{e:diff}
M_T = \whM_T - \ell/2
\end{equation}
asymptotically for large $T$.  

It is
physically obvious that $\whM_T$ being the maximum taken over all
times $t \leq T$ while $M_T$ being the maximum taken at the end of cycles
only, we should have $M_T \leq \whM_T$.  The specific difference
indicated in (\ref{e:diff}) is interesting but a formal proof eludes
us so far.

\appendix

\section{Exact analysis:\ $\ell=2$}
   \label{s:exact2}

We assume that $\varphi_0=1$ and we consider the process $\{Z_t\}$
with transition matrix

\[
P_2= 
U^{2}V^{2}=\left(
\begin{array}
[c]{cccccc}%
\left(  1+2p+3p^{2}\right)  q^{2} & 4p^{3}q & p^{4} & 0 & 0 & \cdots\\
(1+3p)q^{3} & 6p^{2}q^{2} & 4p^{3}q & p^{4} & 0 & \cdots\\
q^{4} & 4pq^{3} & 6p^{2}q^{2} & 4p^{3}q & p^{4} & \cdots\\
0 & q^{4} & 4pq^{3} & 6p^{2}q^{2} & 4p^{3}q & \cdots\\
0 & 0 & q^{4} & 4pq^{3} & 6p^{2}q^{2} & \cdots\\
\vdots & \vdots & \vdots & \vdots & \vdots & \ddots
\end{array}
\right).
\]
The equations for the stationary distribution are as follows:
\begin{align}
   \label{e:zero}
\pi_{0} & =\left(  1+2p+3p^{2}\right)  q^{2}\pi_{0}+(1+3p)q^{3}\pi_{1}+q^{4}%
\pi_{2}, \\
   \label{e:un}
\pi_{1} & =4p^{3}q\pi_{0}+6p^{2}q^{2}\pi_{1}+4pq^{3}\pi_{2}+q^{4}\pi_{3},
  \\
   \label{e:deux}
\pi_{j} & =p^{4}\pi_{j-2}+4p^{3}q\pi_{j-1}+6p^{2}q^{2}\pi_{j}+4pq^{3}\pi
_{j+1}+q^{4}\pi_{j+2}
\qquad \mbox{for $j\geq2$.}
\end{align}
The generating function $F(z)    =
{\displaystyle\sum\limits_{j=2}^{\infty}}
\pi_{j}z^{j}$
may be expressed as
\begin{align*}
F(z)  
  = {} & p^{4}z^{2}%
{\displaystyle\sum\limits_{j=2}^{\infty}}
\pi_{j-2}z^{j-2}+4p^{3}qz%
{\displaystyle\sum\limits_{j=2}^{\infty}}
\pi_{j-1}z^{j-1}+6p^{2}q^{2}%
{\displaystyle\sum\limits_{j=2}^{\infty}}
\pi_{j}z^{j}\\
&  +\frac{4pq^{3}}{z}%
{\displaystyle\sum\limits_{j=2}^{\infty}}
\pi_{j+1}z^{j+1}+\frac{q^{4}}{z^{2}}%
{\displaystyle\sum\limits_{j=2}^{\infty}}
\pi_{j+2}z^{j+2}\\
  = {} &p^{4}z^{2}\left[  F(z)+\pi_{0}+\pi_{1}z\right]  +4p^{3}qz\left[
F(z)+\pi_{1}z\right]  +6p^{2}q^{2}F(z)\\
&  +\frac{4pq^{3}}{z}\left[  F(z)-\pi_{2}z^{2}\right]  +\frac{q^{4}}{z^{2}%
}\left[  F(z)-\pi_{2}z^{2}-\pi_{3}z^{3}\right]
\end{align*}
and we conclude that
\begin{align*}
&  \left[  q^{4}+4pq^{3}z-\left(  1-6p^{2}q^{2}\right)  z^{2}+4p^{3}%
qz^{3}+p^{4}z^{4}\right]  F(z)\\
& \qquad =-p^{4}z^{4}\left(  \pi_{0}+\pi_{1}z\right)  -4p^{3}qz^{3}\left(  \pi
_{1}z\right)  +4pq^{3}z\left(  \pi_{2}z^{2}\right)  +q^{4}\left(  \pi_{2}%
z^{2}+\pi_{3}z^{3}\right)  .
\end{align*}
Replacing $\pi_{2}$ and $\pi_{3}$ by expressions in $\pi_{0}$ and
$\pi_{1}$ from (\ref{e:zero}, \ref{e:un}),
then cancelling the common factor $1-z$ between numerator and denominator,
yields
\[
F(z)=\frac{\left\{  p^{3}(4-3p+pz)\pi_{0}+\left[  -1+6p^{2}-8p^{3}%
+3p^{4}+(4-3p)p^{3}z+p^{4}z^{2}\right]  \pi_{1}\right\}  z^{2}}{(q^{2}%
-p^{2}z)\left[  q^{2}+(1+2pq)z+p^{2}z^{2}\right]  }%
\]
hence%
\[
L=\lim_{z\rightarrow1}F(z)=\frac{2p^{3}(1+q)\pi_{0}-\left[  2(q-p)-q^{4}%
\right]  \pi_{1}}{2(q-p)}.
\]
We observe three zeroes in the denominator $D(z)$ of $F(z)$. \ The first zero,
of smallest modulus $<1$, is negative and given by%
\[
z_{1}=\frac{-1-2pq+\theta}{2p^{2}}%
\]
where $\theta=\sqrt{1+4pq}$. \ The second zero, of intermediate modulus, is
positive and given by%
\[
z_{2}=\frac{q^{2}}{p^{2}} = (\frac{1}{\rho})^2>1.
\]
The third zero, of largest modulus $>1$, is negative and given by%
\[
z_{3}=\frac{-1-2pq-\theta}{2p^{2}}.
\]
Finding the unknowns $\pi_{0}$ and $\pi_{1}$ is achieved by solving two
simultaneous equations:%
{
\[
N(z_1) =0,
\]
found by substituting $z_{1}$ for $z$ in the numerator $N(z)$ for
$F(z)$ and setting this equal to zero, and
\[
\pi_{0}+\pi_{1}+L=1
\]
}
which yields%
\[
\pi_{0}=\frac{(q-p)\left(  3-2p-\theta\right)  }{2q^{4}},
\]%
\[
\pi_{1}=\frac{(q-p)\left[  -1-p-2pq+(1+p)\theta\right]  }{q^{5}}.
\]
Thus we have a complete description of the stationary distribution. \ An exact
expression for $\pi_{j}$ is infeasible; therefore asymptotics as
$j\rightarrow\infty$ are necessary. \ The second zero $z_{2}$ leads, by
classical singularity analysis \cite{FS-heu} to%
\[
A(p)=-\frac{N(z_{2})}{z_{2}D^{\prime}(z_{2})}=\frac{(q-p)\left[
1+(q-p)\theta\right]  }{4q^{4}},
\]%
\[
\pi_{j}\sim A(p)\left(  \dfrac{p^{2}}{q^{2}}\right)  ^{j}
\]
and this confirms (\ref{e:piklim}).
This is the expression that we shall employ in the clumping heuristic. \ 

Consider the random walk on the integers with transition matrix
$\widetilde P_2$ (see Equation (\ref{e:ptilde})).  For any $i$, the
random walk jumps to $i+j$, $-2 \leq j \leq 2$ with probability $p_j$
defined in (\ref{e:pj}).

For nonzero $j$, let $\nu_{j}$ denote the probability that, starting from
$-j$, the walker eventually hits $0$.  For $j=0$, $\nu_{0}$ is the probability
that, starting from $0$, the walker eventually returns to $0$.
We have  
\begin{align}
   \label{e:nu0}
\nu_{0}
  &=p^{4}\nu_{-2}+4p^{3}q\nu_{-1}+6p^{2}q^{2}+4pq^{3}\nu_{1}+q^{4}\nu_{2},
\\
   \label{e:nu}
\nu_{j} & =p^{4}\nu_{j-2}+4p^{3}q\nu_{j-1}+6p^{2}q^{2}\nu_{j}+4pq^{3}\nu
_{j+1}+q^{4}\nu_{j+2} \quad j\geq 1,
\end{align}
by a simple argument 
(note that $\nu_0$ is replaced by 1 in (\ref{e:nu}))
and
the generating function 
$\tilde{F}(z)   =
\sum_{j=1}^{\infty}
\nu_{j}z^{j}
$
is expressed as
\begin{align*}
\tilde{F}(z) 
  = {} &p^{4}z^{2}%
{\displaystyle\sum\limits_{j=1}^{\infty}}
\nu_{j-2}z^{j-2}+4p^{3}qz%
{\displaystyle\sum\limits_{j=1}^{\infty}}
\nu_{j-1}z^{j-1}+6p^{2}q^{2}%
{\displaystyle\sum\limits_{j=1}^{\infty}}
\nu_{j}z^{j}\\
&  +\frac{4pq^{3}}{z}%
{\displaystyle\sum\limits_{j=1}^{\infty}}
\nu_{j+1}z^{j+1}+\frac{q^{4}}{z^{2}}%
{\displaystyle\sum\limits_{j=1}^{\infty}}
\nu_{j+2}z^{j+2}  \qquad \qquad \mbox{by (\ref{e:nu}),}\\
  = {} &p^{4}z^{2}\left[  \tilde{F}(z)+\nu_{-1}z^{-1}+1\right]  +4p^{3}qz\left[
\tilde{F}(z)+1\right]  +6p^{2}q^{2}\tilde{F}(z)\\
&  +\frac{4pq^{3}}{z}\left[  \tilde{F}(z)-\nu_{1}z\right]  +\frac{q^{4}}%
{z^{2}}\left[  \tilde{F}(z)-\nu_{1}z-\nu_{2}z^{2}\right].
\end{align*}
Equivalently,
\begin{align}
   \nonumber 
  (1-z) & (q^{2}-p^{2}z)\left[  q^{2}+(1+2pq)z+p^{2}z^{2}\right]
          \tilde{F}(z)\\
   \nonumber 
  = {} & -p^{4}z^{3}\nu_{-1}-p^{4}z^{4}-4p^{3}qz^{3}+4pq^{3}z^{2}\nu_{1}+q^{4}%
z\nu_{1}\\
   \nonumber 
&  +z^{2}\left(  \nu_{0}-p^{4}\nu_{-2}-4p^{3}q\nu_{-1}-6p^{2}q^{2}-4pq^{3}%
\nu_{1}\right)  \\
   \label{e:num}
  = {} &z^{2}\nu_{0}+q^{4}z\nu_{1}-p^{4}z^{3}\nu_{-1}-4p^{3}qz^{2}\nu_{-1}%
-p^{4}z^{2}\nu_{-2}-6p^{2}q^{2}z^{2}-4p^{3}qz^{3}-p^{4}z^{4}.
\end{align}
Only the first two of the four zeroes $z_{1}$, $1$, $z_{2}$, $z_{3}$
are of interest. \ Let $\tilde{N}(z)$ denote the numerator for
$\tilde{F}(z)$, {that is, $\tilde{N}(z)$ is the expression on the
  right-hand side of (\ref{e:num}).  We have
\begin{equation}
   \label{e:eq1}
\tilde{N}(z_1)  =0,
\qquad 
\tilde{N}(1)  =0.
\end{equation}
}
Using%
\[%
\begin{array}
[c]{ccc}%
\nu_{-j}=p^{4}\nu_{-j-2}+4p^{3}q\nu_{-j-1}+6p^{2}q^{2}\nu_{-j}+4pq^{3}%
\nu_{-j+1}+q^{4}\nu_{-j+2}, &  & j\geq1;
\end{array}
\]
we deduce that%
\[
\nu_{-j}=\nu_{j}\left(  \frac{q^{2}}{p^{2}}\right)  ^{j}%
\]
since multiplying both sides of%
\begin{align*}
\nu_{j}\left(  \frac{q^{2}}{p^{2}}\right)  ^{j}  & =p^{4}\nu_{j+2}\left(
\frac{q^{2}}{p^{2}}\right)  ^{j+2}+4p^{3}q\nu_{j+1}\left(  \frac{q^{2}}{p^{2}%
}\right)  ^{j+1}+6p^{2}q^{2}\nu_{j}\left(  \frac{q^{2}}{p^{2}}\right)  ^{j}\\
& +4pq^{3}\nu_{j-1}\left(  \frac{q^{2}}{p^{2}}\right)  ^{j-1}+q^{4}\nu
_{j-2}\left(  \frac{q^{2}}{p^{2}}\right)  ^{j-2}%
\end{align*}
by $p^{2j}/q^{2j}$ gives an identity. \ Replacing $q^{4}\nu_{2}$ by $p^{4}%
\nu_{-2}$ in our Equation (\ref{e:nu0}) for $\nu_{0}$ gives%
\begin{equation}
   \label{e:eq3}
\nu_{0}=2p^{4}\nu_{-2}+4p^{3}q\nu_{-1}+6p^{2}q^{2}+4pq^{3}%
\nu_{1}.
\end{equation}
Also, replacing $q^{2}\nu_{1}$ by $p^{2}\nu_{-1}$ {in 
  Equations (\ref{e:eq1}, \ref{e:eq3})} reduces the number of variables to
three. \ The simultaneous solution is
\[
\nu_{0}=\frac{-1+2p+8p^{2}-8p^{3}+(q-p)^{2}\theta}{4pq},
\]%
\[
\nu_{-1}=\frac{1-8p^{2}+16p^{3}-8p^{4}-(q-p)\theta}{8p^{3}q},
\]%
\[
\nu_{-2}=\frac{-1-2p+12p^{2}-24p^{4}+24p^{5}-8p^{6}+(q-p)(1+2p-4p^{2})\theta
}{8p^{5}q}%
\]
yielding%
\[
\nu_{1}=\frac{1-8p^{2}+16p^{3}-8p^{4}-(q-p)\theta}{8pq^{3}}%
\]
in particular. \ 
{The root $z_2$ may be used to obtain the asymptotics of $\nu_j$
for $j$ large.}

Readers might be tempted to use the level $k$ as the
absorbing set $S$. But the maximum could be above $k$ without ever touching
level $k$ because of the transition $p^{4}$. So we must use as $S$ the levels
$k$ and $k+1$: no maximum can be above $k+1$ without touching at least one of
the levels $k$ or $k+1$. In the revised notation, this leads to the absorbing
set $\Omega=\{0,-1\}$.

An idea of Aldous \cite{Ald-heu} now comes crucially into play. \ The rate
$\lambda$ of clumps of visits to $\Omega$ is equal to $\lambda_{0}%
+\lambda_{-1}$ where parameters $\lambda_{0}$ and $\lambda_{-1}$ are solutions
of the system%
\[
\lambda_{0}+\lambda_{-1}\nu_{-1}=\left(  1-\nu_{0}\right)  \pi_{j},
\]%
\[
\lambda_{0}\nu_{1}+\lambda_{-1}=\left(  1-\nu_{0}\right)  \pi_{j+1}\sim
\frac{p^{2}}{q^{2}}\left(  1-\nu_{0}\right)  \pi_{j}.
\]
In words, for nonzero $j$, the ratio $\nu_{j}/\left(  1-\nu_{0}\right)  $ is
the expected sojourn time in $\{0\}$, given that the walk started at $-j$.
\ The total clump rate is consequently%
\[
\lambda\sim\frac{(q-p)\left[  1+(q-p)\theta\right]  }{2q^{2}}\pi_{j}%
\]
in association with the transition matrix $U^{2}V^{2}$, that is, the sub-walk
for $\varphi_0=1$.

The total clump rate for $\varphi_0=3$, that is, the transition matrix
$V^{2}U^{2}$,
can similarly be shown to be%
\[
\lambda^{\prime}\sim\frac{(q-p)\left[  1+(q-p)\theta\right]  }{2p^{2}}\pi_{j}%
\]
and this particular sub-walk clearly contains the full walk maximum. \ Of
course, the four sub-walk maxima are not independent.

Note, if $j=\log_{q^{2}/p^{2}}(n)+h+1$, we have%
\[
\left(  \dfrac{q^{2}}{p^{2}}\right)  ^{j}=n\left(  \dfrac{q^{2}}{p^{2}%
}\right)  ^{h+1}%
\]
thus%
\[
\pi_{j}\frac{n}{4}\sim\frac{A(p)}{4}\left(  \dfrac{p^{2}}{q^{2}}\right)
^{j}n=\frac{A(p)}{4}\left(  \dfrac{q^{2}}{p^{2}}\right)  ^{-(h+1)}%
=\frac{(q-p)\left[  1+(q-p)\theta\right]  }{16q^{4}}\,\frac{p^{2}}{q^{2}%
}\left(  \dfrac{q^{2}}{p^{2}}\right)  ^{-h}.
\]
By the clumping heuristic, the desired exponential argument is%
\[
\frac{\lambda^{\prime}}{\pi_{j}}\cdot\pi_{j}\frac{n}{4}\sim\frac
{(q-p)^{2}\left[  1+(q-p)\theta\right]  ^{2}}{32q^{6}}\left(  \dfrac{q^{2}%
}{p^{2}}\right)  ^{-h}=\frac{\chi_{2}(p)}{4}\left(  \dfrac{q^{2}}{p^{2}%
}\right)  ^{-h}%
\]
as was to be shown.

\section{Explicit forms for $G$ and $H$}
\label{s:GH}

We define $g_j = G_{1,\ell-j+1}$, $1 \leq j \leq \ell$ that is, $g_j$ is the
probability of visiting $n+1- j$, starting from $n + 1$,
before any other state with index $s \leq n$, and one proves in
\cite[Section 5.5]{blm05} that 
\[
G = \vligne{
0 & 1 \\
& \ddots & \ddots \\
 &  & 0 & 1 \\
g_\ell & \ldots  & g_2 & g_1}^\ell
\]
or $G=C(\vg)^\ell$ with $\vg =
\vligne{g_\ell & \ldots g_1}\tr$.  Furthermore, the eigenvalues of
$C(\vg)$ being the roots of the polynomial $g(z)= z^\ell - \sum_{\leq i \leq \ell - 1}
g_{\ell - i} z^i$, we have $g(z)= \prod_{1 \leq i \leq \ell}
(z-\sigma_i)$.

We obtain the expression for $H$ by a similar argument.  We define
$h_j = H_{\ell j}$, $1 \leq j \leq \ell$, that is, $h_j$ is the
probability of visiting $n+\ell+j$, starting from $n+\ell$, before any other
state with index $s \geq n+\ell+1$, independently of $n$.  For $i \leq
\ell-1$, we decompose the paths from $n+i$ to $n+\ell+j$ in two sets,
depending on whether they go through $n+\ell$ or not, and obtain
\[
H_{ij} = H_{i+1,j+1} + H_{i+1,1} h_j, \qquad 1 \leq i \leq \ell -1,
\]
and so conclude that $H = \widehat C(\vh)^\ell $, with $\vh =
\vligne{h_1 & h_2 & \ldots & h_\ell }\tr$, where the matrix $\widehat C$
is defined as 
\[
\widehat C(\vx) = 
\vligne{x_0 & x_1 & \ldots & x_{\ell-1} \\
1 & 0 \\
 & \ddots & \ddots \\
& & 1 & 0
}
\]
for $\vx = \vligne{x_0 & \ldots & x_{\ell-1}}\tr$.  

The eigenvalues of $C(\vh)$ are the roots of the polynomial $z^\ell -
\sum_{0 \leq i \leq \ell -1} h_{\ell-i} z^i$, on the one hand, and the
eigenvalues of $H$ and $R$ are the same (by \cite[Theorem
3.20]{blm05}) on the other hand, from which we conclude that 
$h_{\ell-j} = r_j$, for $0 \leq j \leq \ell-1$, and
\[
H =
\vligne{r_{\ell-1} & r_{\ell-2 }& \ldots & r_0 \\
1 & 0 \\
 & \ddots & \ddots \\
& & 1 & 0
}^\ell.
\]

\subsection*{Acknowledgments}

Guy Latouche and Beatrice Meini acknowledge the financial support of
GNCS of INdAM, Italy and of ``Fondi di ateneo'' of the University of Pisa.


\begin{figure}[tb]
\includegraphics[scale=0.9]{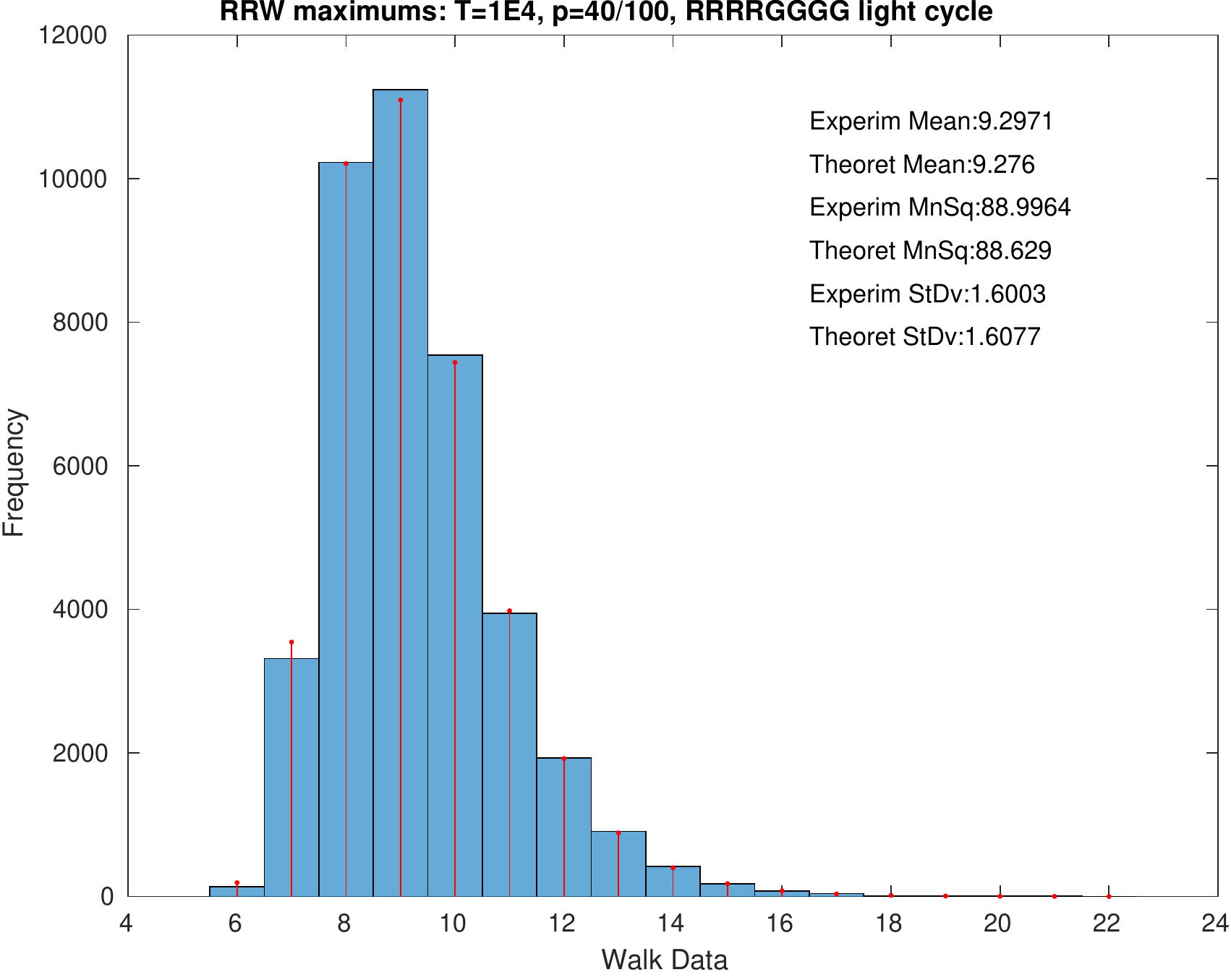} 
\caption{
  \label{f:mt1}
Comparison of the heuristic distribution of $M_T$ and its density
obtained by simulation.  The number of replications is $40,000$, the
parameters are $T= 10^4$, $\ell = 4$ and $p=0.40$.
}
\end{figure}

\begin{figure}[tb]
\includegraphics[scale=0.9]{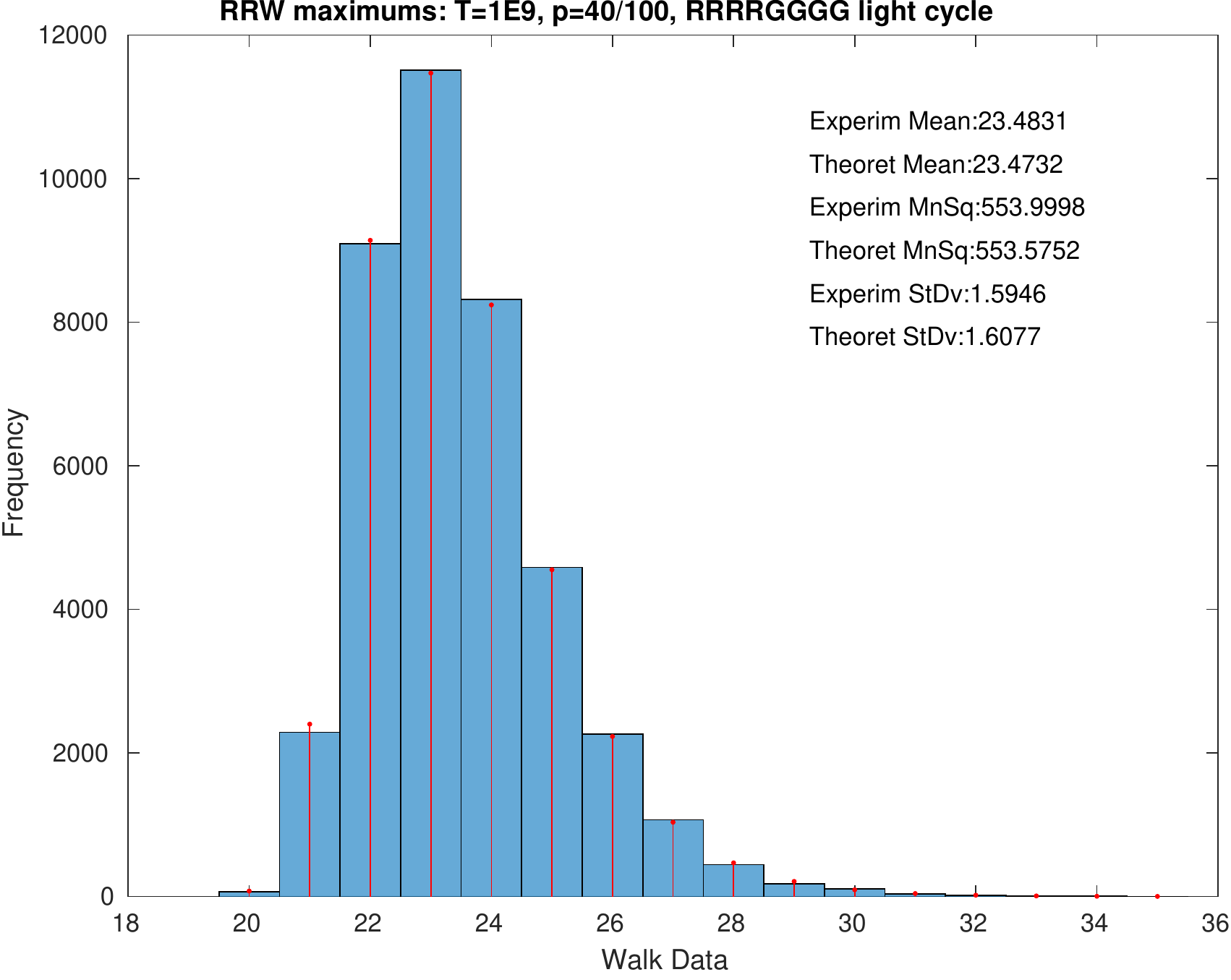} 
\caption{
  \label{f:mt2}
Comparison of the heuristic distribution of $M_T$ and its density
obtained by simulation.  The number of replications is $40,000$, the
parameters are $T= 10^9$, $\ell = 4$ and $p=0.40$.
} 
\end{figure}

\begin{figure}[tb]
\includegraphics[scale=0.9]{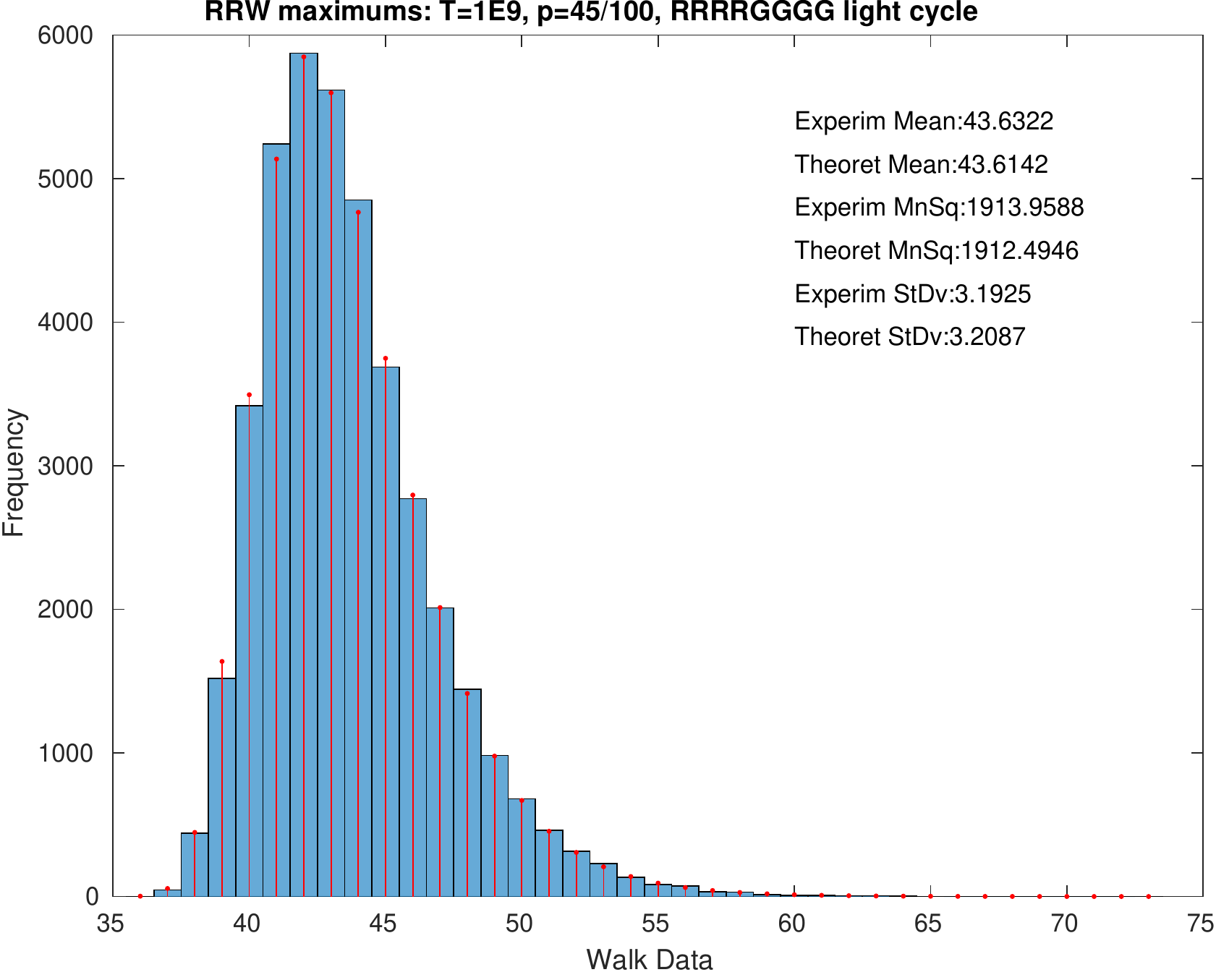} 
\caption{
  \label{f:mt3}
Comparison of the heuristic distribution of $M_T$ and its density
obtained by simulation.  The number of replications is $40,000$, the
parameters are $T= 10^9$, $\ell = 4$ and $p=0.45$.
} 
\end{figure}

\end{document}